\documentclass[12pt, a4paper]{article}
\usepackage{amsfonts}
\usepackage{mathrsfs}
\usepackage{latexsym}
\usepackage{xy}
\usepackage{amsfonts,amsmath,amssymb,amsthm}
\usepackage{color}

\xyoption{all}

\newcommand{\bcen}{\begin{center}}     \newcommand{\ecen}{\end{center}}
\newcommand{\bay}{\begin{array}}      \newcommand{\eay}{\end{array}}
\newcommand{\beq}{\begin{eqnarray*}}      \newcommand{\eeq}{\end{eqnarray*}}

\def\fd{\mathrm{fin.dim}}

\def\Hom{\mathrm{Hom}}
\def\Ker{\mathrm{Ker}}
\def\Im{\mathrm{Im}}
\def\Ext{\mathrm{Ext}}
\def\Tor{\mathrm{Tor}}

\def\mod{\mathrm{mod}}
\def\Mod{\mathrm{Mod}}
\def\id{\mathrm{id}}
\def\Id{\mathrm{Id}}
\def\pd{\mathrm{pd}}
\def\fd{\mathrm{fd}}

\def\proj{\mathrm{proj}}
\def\Gproj{\mathrm{Gproj}}
\def\GProj{\mathrm{GProj}}
\def\Proj{\mathrm{Proj}}

\def\Coker{\mathrm{Coker}}

\def\Add{\mathrm{Add}}

\begin{document}

\newtheorem{theorem}{Theorem}[section]
\newtheorem{proposition}[theorem]{Proposition}
\newtheorem{lemma}[theorem]{Lemma}
\newtheorem{corollary}[theorem]{Corollary}
\newtheorem{remark}[theorem]{Remark}
\newtheorem{example}[theorem]{Example}
\newtheorem{definition}[theorem]{Definition}
\newtheorem{question}[theorem]{Question}
\numberwithin{equation}{section}

\title{\large\bf
Gorenstein projective
objects over cleft extensions}

\author{\large Yongyun Qin}

\date{\footnotesize School of Mathematics, Yunnan Key Laboratory of Modern Analytical Mathematics and Applications,
Yunnan Normal University, Kunming, Yunnan 650500, China.
\\E-mail: qinyongyun2006@126.com
	}

\maketitle

\begin{abstract}
In this paper we introduce compatible
cleft extensions of abelian categories, and we prove that if
$(\mathcal{B},\mathcal{A},
e,i,l)$ is a compatible cleft extension, then both the functor
$l$ and the left adjoint of $i$ preserve Gorenstein projective objects.
Moreover, we give some necessary conditions for an object of $\mathcal{A}$ to be
Gorenstein projective, and we show that these necessary conditions are also sufficient in some special case.
As applications, we unify some known results on the description of Gorenstein projective modules
over triangular matrix rings, Morita context rings with zero homomorphisms and $\theta$-extensions.
\end{abstract}

\medskip


\medskip


\bigskip

\section{\large Introduction}

\indent\indent
The main idea of Gorenstein homological algebra is to replace projective modules by Gorenstein projective
modules. This kind of modules traces back to Auslander and Bridger \cite{AB69}, where
they study modules of G-dimension zero over noetherian rings.
After two decades, Enochs and Jenda generalized these modules to Gorenstein projective modules
over an arbitrary ring \cite{EJ95}. A fundamental problem in Gorenstein homological
algebra is determining all the Gorenstein-projective
modules for a given ring. However, this task is difficult
and was only investigated
for some special kinds of rings, such as triangular matrix rings \cite{LZ20, XZ12, Z13}, Morita context rings
with zero homomorphisms \cite{Ase22, GP17, GX24},
trivial ring extensions
\cite{M25} and tensor rings \cite{CL20, DLST25}.
A natural generalization of these rings is the class of $\theta$-extension rings
introduced by Marmaridis \cite{Mar93}.
On the other hand, there are many works on how the property of being
Gorenstein-projective can be preserved under certain functors, see \cite{CR22, LHZ22,
Lu17, Lu19, QS24}

A {\it cleft extension} of an abelian category $\mathcal{B}$
is an abelian category $\mathcal{A}$
together with functors $$\xymatrix@!=4pc{ \mathcal{B} \ar[r]|{i} & \mathcal{A}
			 \ar[r]|{e} & \mathcal{B}
			\ar@/_1pc/[l]|{l} }  $$
such that the functor $e$ is faithful exact and admits a left adjoint $l$, and there is a natural isomorphism
$ei \cong \Id_{\mathcal{B}}$. We denote this cleft extension by $(\mathcal{B},\mathcal{A},
e,i,l)$, which give rise to a functor $q$ such that $(q,i)$ is an adjoint pair. Also, there are
two endofunctors $F: \mathcal{B}\rightarrow \mathcal{B}$ and $G: \mathcal{A}\rightarrow \mathcal{A}$
with certain properties, see Section~\ref{Section-definitions and conventions}.
The concept of cleft extension was
introduced by Beligiannis \cite{Bel00} as a generalization of trivial extension of abelian categories
defined by Fossum-Griffith-Reiten \cite{FGR75}.
Recently, a theorem of Kostas-Psaroudakis stating that the cleft extensions of module categories occur
precisely as $\theta$-extensions of rings \cite{KP25}. Therefore, the cleft extension of abelian category
provide a common framework for the study of trivial extension rings
and tensor rings. Moreover, the framework of cleft extension was used to reduce some homological properties
of rings, such as the finitistic dimension \cite{EGPS22, GPS18}, cotorsion pairs \cite{H25},
the property
of injective generation \cite{KP25}, the Igusa-Todorov distances and
extension dimensions \cite{MZL25}, the Gorenstein weak global (flat-cotorsion) dimension
and relative singularity categories \cite{LMY25},
the Gorenstein projective modules, singularity categories and Gorensteinness \cite{EPS22, Kos24}.
In particular, Kostas proved that both $l$ and $q$ preserve
Gorenstein projective objects if $F$ is perfect and nilpotent \cite{Kos24}. The main objective
of this paper is to extend the result of Kostas to
a more general setting.
To achieve this goal, we introduce the conception of compatible cleft extensions (see Definition~\ref{def-comp-exten}),
which is motivated by the
connection with the compatible bimodules, see Proposition~\ref{tri-matri-comp}.
Our first result is the following, which is listed in Theorem~\ref{thm-comp-exten}.

\medskip

{\bf Theorem I.} {\it
Let $(\mathcal{B}, \mathcal{A}, e, i, l)$ be a cleft extension which is compatible.
Then both $l$ and $q$ preserve Gorenstein projective objects.

}

\medskip

We prove that if $F$ is perfect and nilpotent in the sense of Kostas \cite{Kos24} then the
cleft extension is compatible, and the converse may not be true,
see Proposition~\ref{proposition-perfect} and Remark~\ref{remark-comp-exten}. Therefore,
in order to restrict $q$ and $l$ to the subcategories of Gorenstein projective objects, our compatibility conditions on cleft extension
are more general than that of Kostas \cite{Kos24}.
Moreover, applying Theorem I, we give some necessary conditions for an object of $\mathcal{A}$ to be
Gorenstein projective, and we show that these necessary conditions are also sufficient in some special case.

\medskip

{\bf Theorem II.} (See Theorem~\ref{theorem-Gpro-condition})
{\it Let $(\mathcal{B}, \mathcal{A}, e, i, l)$ be a cleft extension which is compatible. Take $X\in \mathcal{A}$
and consider
the following conditions:

{\rm (1)} $X \in \GProj\mathcal{A}$;

{\rm (2)} $q(X) \in \GProj\mathcal{B}$ and $q(u _X)$ is a monomorphism;

{\rm (3)} $q(X) \in \GProj\mathcal{B}$ and the sequence $\xymatrix{F^2e(X)
\ar[r]^{\alpha _X} & Fe(X) \ar[r]^{\beta _X} & e(X) } $ is exact.

Then we have {\rm (1)} $\Rightarrow$ {\rm (2)} $\Leftrightarrow$ {\rm (3)}.
Moreover, all three conditions are equivalent if the corresponding natural transformation $\eta : F^2\rightarrow F$ is zero.
}

As an application of Theorem II, we unify some known results on the description of (infinitely generated) Gorenstein projective modules
over triangular matrix rings, Morita context rings with zero homomorphisms and $\theta$-extensions, see
Corollary~\ref{cor-theta-extension}, Corollary~\ref{cor-Morita-ring} and Corollary~\ref{cor-tri-alg}.

The paper is organized as follows. In section 2, we recall some relevant
definitions and conventions.
In section 3 we prove Theorem I and Theorem II.
In section 4, we give some applications
in triangular matrix rings, Morita context rings, trivial ring extensions
and $\theta$-extensions.

 \section{\large Definitions and conventions}\label{Section-definitions and conventions}

\indent\indent In this section, we give some notions and some preliminary results.

Let $\mathcal{A}$ be an abelian category and the subcategories discussed in this paper are full
and closed under isomorphisms. We use $\Proj \mathcal{A}$ to denote the subcategories of $\mathcal{A}$
consisting of all projective objects. As usual, we write the following (co)complex of $\mathcal{A}$
$$\xymatrix{\cdots \ar[r]& X^{-1}
\ar[r]^{d^{-1}} & X^0 \ar[r]^{d^{0}} & X^1
\ar[r] & \cdots} $$
as $(X^\bullet, d^\bullet)$,
where $d^i:X^i\rightarrow X^{i+1}$ is the $i$-th differential. Sometimes, for simplicity, we shall write
it as $X^\bullet$ without mentioning the morphisms $d^i$.

Recall that a {\it complete $\mathcal{A}$-projective resolution} is an exact sequence
$$\xymatrix{(P^\bullet,d^\bullet) = \cdots \ar[r]& P^{-1}
\ar[r]^{d^{-1}} & P^0 \ar[r]^{d^{0}} & P^1
\ar[r] & \cdots} $$
in $\mathcal{A}$ with all $P^i\in \Proj \mathcal{A}$, such that $\Hom _A(P^\bullet, Q)$ is exact for
every $Q \in \Proj \mathcal{A}$.

\begin{definition}{\rm(\cite{EJ00})
 An object $G \in \mathcal{A}$ is called {\it Gorenstein projective} if there exists a complete
$\mathcal{A}$-projective resolution $(P^\bullet, d^\bullet)$ such that $G = \Im(d^{0})$.}

\end{definition}

Denote by $\GProj \mathcal{A}$
the subcategory of $\mathcal{A}$ consisting of all Gorenstein projective objects.
It is well known that $\GProj \mathcal{A}$
is a Frobenius category such that each object in $\Proj \mathcal{A}$ is projective-injective and
its stable category $\underline{\GProj \mathcal{A}}$ is a triangulated category.

Now let's recall and review some results about cleft extensions
of abelian categories from \cite{Bel00, GPS18}.

\begin{definition}{\rm (\cite{Bel00})
A {\it cleft extension} of an abelian category $\mathcal{B}$
is an abelian category $\mathcal{A}$
together with functors $$\xymatrix@!=4pc{ \mathcal{B} \ar[r]|{i} & \mathcal{A}
			 \ar[r]|{e} & \mathcal{B}
			\ar@/_1pc/[l]|{l} }  $$
such that the functor $e$ is faithful exact and admits a left adjoint $l$, and there is a natural isomorphism
$ei \cong \Id_{\mathcal{B}}$.
}
\end{definition}

From now on we will denote a cleft extension by $(\mathcal{B},\mathcal{A},
e,i,l)$, and it
give rise to additional properties. For instance, it follows that the functor $i$ is fully
faithful and exact (see \cite[Lemma 2.2]{Bel00}, \cite[Lemma 2.2(ii)]{GPS18}). Moreover, there is a
functor $q: \mathcal{A} \rightarrow \mathcal{B}$, which is left adjoint of $i$ (see \cite[Proposition 2.3]{Bel00}
and \cite[Lemma 2.2(iv)]{GPS18}). Then, $(ql, ei)$ is an adjoint pair and since $ei \cong \Id_{\mathcal{B}}$,
it follows that $ql \cong \Id_{\mathcal{B}}$. Moreover, there are endofunctors
$F: \mathcal{B}\rightarrow \mathcal{B}$ and $G: \mathcal{A}\rightarrow \mathcal{A}$ that
appear in the following short exact sequences
\begin{equation}\label{G}
0\rightarrow G\rightarrow le\rightarrow \Id_{\mathcal{A}}\rightarrow 0
\end{equation}
and
\begin{equation}\label{F}
0\rightarrow F\rightarrow el\rightarrow \Id_{\mathcal{B}}\rightarrow 0,
\end{equation}
where the sequence {\rm (\ref{F})} splits (see \cite[Section 2]{GPS18}).
The functor $F$ is defined as $eGi$ and from the above
one can prove that $F^ne \cong eG^n$, thus in particular $F$ is nilpotent if and only if $G$ is
nilpotent (see \cite[Lemma 2.4]{GPS18}). It also follows that for every
$A \in \mathcal{A}$ and $n \geq 1$, there is a short exact sequence
\begin{equation}\label{Gn}
0\rightarrow G^n(A)\rightarrow lF^{n-1}e(A)\rightarrow G^{n-1}(A)\rightarrow 0,
\end{equation}
see \cite{GPS18} or \cite[Lemma 2.2]{MZL25}.

Throughout this paper, all rings are nonzero associative rings with identity and all
modules are unitary. For such a ring $A$, we denote by $\Mod A$ the category of left $A$-modules, and by $\mod A$,
$\Proj A$ and $\proj A$ its full subcategories consisting of all finitely generated modules,
projective modules and finitely generated projective modules, respectively.
Similarly, we have the notations $\GProj A$ (resp. $\Gproj A$) for the full subcategories of all
the (resp. finitely generated) Gorenstein projective modules.
For any $X\in \Mod A$, we write $\Add X$ for the class of $A$-modules isomorphic to direct summands of direct sums
of copies of $X$, and we denote by $\pd(X), \id(X)$ and $\fd(X)$ the projective,
injective and flat dimensions of $X$, respectively.

\section{Gorenstein projective objects over cleft extensions}
\indent\indent
Let $(\mathcal{B},\mathcal{A},
e,i,l)$ be a cleft extension of abelian categories. Then
both $l$ and $q$ preserve projective objects since their
right adjoint $i$ and $e$ are exact.
In this section, we will
observe when the functors $l$ and $q$ preserve
Gorenstein projective objects, and we will give some necessary conditions for an object of $\mathcal{A}$ to be
Gorenstein projective. Recall there are two endofunctors
$F: \mathcal{B}\rightarrow \mathcal{B}$ and $G: \mathcal{A}\rightarrow \mathcal{A}$
associated with the cleft extension $(\mathcal{B},\mathcal{A},
e,i,l)$.

\begin{proposition}\label{proposition-l-preserve-Gproj}
Let $(\mathcal{B},\mathcal{A},
e,i,l)$ be a cleft extension of abelian categories. Assume that for any complete $\mathcal{B}$-projective resolution $P^\bullet$
and any $P'\in \Proj \mathcal{B}$,
both the complexes $F(P^\bullet)$ and $\Hom _{\mathcal{B}}(P^\bullet , F(P')) $ are acyclic. Then
the functor $l$ preserves
Gorenstein projective objects.
\end{proposition}
\begin{proof}
Let $X \in \GProj\mathcal{B}$ and let $(P^\bullet,d^\bullet)$ be a complete $\mathcal{B}$-projective resolution
with $X=\Im (d^{0})$. Then $l(P^\bullet)$ is a complex with all terms in $\Proj\mathcal{A}$.
On the other hand, it follows from (\ref{F}) that there is a short exact sequence of complexes
$$0\rightarrow F(P^\bullet)\rightarrow el(P^\bullet)\rightarrow P^\bullet \rightarrow 0.$$
Since $F(P^\bullet)$ and $P^\bullet$ are exact, we infer that
the complex $el(P^\bullet)$ is acyclic.
Further, it follows from the exactness of $e$ that $0=H^i(el(P^\bullet))\cong e(H^i(l(P^\bullet)))$ for any $i\in \mathbb{Z}$,
and thus $H^i(l(P^\bullet))=0$ since $e$ is faithful. Therefore, the complex $l(P^\bullet)$
is exact, and to show $l(X)\in \GProj\mathcal{A}$, it suffices to prove that
$\Hom _{\mathcal{A}}(l(P^\bullet) , Q)$ is exact for
every $Q \in \Proj \mathcal{A}$. Indeed, we have an adjoint isomorphism $\Hom _{\mathcal{A}}(l(P^\bullet), Q)\cong
\Hom _{\mathcal{B}}(P^\bullet , e(Q))$, and by
\cite[Lemma 2.3]{Kos24}, we get that $Q$ is a direct summand of $l(P')$ for some $P' \in \Proj \mathcal{B}$.
So it suffices to show that $\Hom _{\mathcal{B}}(P^\bullet , el(P'))$ is exact.
Applying $\Hom _{\mathcal{B}}(P^\bullet , -)$ to the exact sequence
$$0\rightarrow F(P')\rightarrow el(P')\rightarrow P' \rightarrow 0,$$
we get a short exact
sequence of complexes $$0\rightarrow \Hom _{\mathcal{B}}(P^\bullet , F(P'))\rightarrow
\Hom _{\mathcal{B}}(P^\bullet , el(P')) \rightarrow \Hom _{\mathcal{B}}(P^\bullet , P') \rightarrow 0.$$
Since $P^\bullet$ is a complete $\mathcal{B}$-projective resolution, we infer that $\Hom _{\mathcal{B}}(P^\bullet , P')$
is exact, and by assumption, we have that $\Hom _{\mathcal{B}}(P^\bullet , F(P'))$ is acyclic.
As a result, we obtain that $\Hom _{\mathcal{B}}(P^\bullet , el(P'))$ is also acyclic.
\end{proof}

\begin{proposition}\label{proposition-q-preserve-Gproj}
Let $(\mathcal{B},\mathcal{A},
e,i,l)$ be a cleft extension of abelian categories.

{\rm (i)} Assume that for any complete $\mathcal{A}$-projective resolution $Q^\bullet$
and any $P\in \Proj \mathcal{B}$,
both the complexes $q(Q^\bullet)$ and $\Hom _{\mathcal{A}}(Q^\bullet , i(P)) $ are acyclic.
Then the functor $q$ preserves
Gorenstein projective objects;

{\rm (ii)} If $F$ is nilpotent, $\mathbb{L}_iF^j(F(P))=0$ for any $i,j>0$ and $P\in \Proj \mathcal{B}$,
and for any complete $\mathcal{A}$-projective resolution $Q^\bullet$,
any $P\in \Proj \mathcal{B}$ and any $j>0$,
both the complexes $F^j(eQ^\bullet)$ and $\Hom _{\mathcal{A}}(Q^\bullet , lF^j(P)) $ are acyclic. Then the assumptions
in {\rm (1)} are satisfied and then the functor $q$ preserves
Gorenstein projective objects.
\end{proposition}
\begin{proof}
(i) Let $X \in \GProj\mathcal{A}$ and let $(Q^\bullet,d^\bullet)$ be a complete $\mathcal{A}$-projective resolution
with $X=\Im (d^{0})$. Then $q(Q^\bullet)$ is an exact complex with all terms in $\Proj\mathcal{B}$.
Therefore, to show $q(X)\in \GProj\mathcal{B}$, it suffices to prove that
$\Hom _{\mathcal{B}}(q(Q^\bullet), P)$ is exact for
every $P \in \Proj \mathcal{B}$, and this follows from
the adjoint isomorphism $\Hom _{\mathcal{B}}(q(Q^\bullet), P)\cong \Hom _{\mathcal{A}}(Q^\bullet, i(P))$
and the assumption that $\Hom _{\mathcal{A}}(Q^\bullet , i(P))$ is acyclic.

(ii) Assume that there exists some integer $s$ such that $F^s=0$.
Let $X \in \Gproj\mathcal{A}$ and let $(Q^\bullet,d^\bullet)$ be a complete $\mathcal{A}$-projective resolution
with $X=\Im (d^{0})$. First, we claim that
$\mathbb{L}_iq(X)=0$ for any $i>s$.
For this, consider the acyclic complex $(e(Q^\bullet),e(d^\bullet))$ with $e(X)=\Im (e(d^{0}))$.
Since $\mathbb{L}_iF^j(F(P))=0$ for any $i,j>0$ and $P\in \Proj \mathcal{B}$, it follows from
\cite[Lemma 3.6]{Kos24} that $\mathbb{L}_iF^j(e(Q^s)))=0$ for any $i,j>0$ and $s\in \mathbb{Z}$.
Therefore, the complex $F^j(eQ^\bullet)$
computes $\mathbb{L}_iF^j(e(X))$, that is, $\mathbb{L}_iF^j(e(X))\cong H^{-i}(F^j(e(Q^\bullet)))$,
for any $i,j>0$. By assumption, we have that $F^j(eQ^\bullet)$ is acyclic, and then
$\mathbb{L}_iF^j(e(X))=0$ for any $i,j>0$. Applying
\cite[Lemma 3.3]{Kos24}, we obtain $\mathbb{L}_iF(F^j(e(X)))=0$ for any $i\geq 1$ and $j\geq0$,
and by \cite[Lemma 2.2 (ii)]{Kos24}, we get $\mathbb{L}_il(F^j(e(X)))=0$ for any $i\geq 1$ and $j\geq0$.
This shows that in order to compute $\mathbb{L}_iq(lF^j(e(X)))$, it is enough to begin with a projective resolution of
$F^j(e(X))$ and
apply $ql$. But $ql \cong \Id _{\mathcal{B}}$ and then we deduce that $\mathbb{L}_iq(lF^j(e(X)))=0$
for any $i\geq 1$ and $j\geq0$.
Now consider the
the following
short exact sequences arising from (\ref{Gn})
$$0\rightarrow G(X)\rightarrow le(X)\rightarrow X\rightarrow 0, 0\rightarrow G^2(X)\rightarrow lFe(X)\rightarrow G(X)\rightarrow 0, \cdots$$
which yield the following isomorphisms for any $i>t$
$$\mathbb{L}_iq(X)\cong \mathbb{L}_{i-1}q(G(X))\cong \mathbb{L}_{i-2}q(G^2(X))
\cong \cdots \cong \mathbb{L}_{i-t}q(G^t(X)).$$
Since $F^s=0$ and $F^se=eG^s$, we have $G^s=0$ and
for any $i>s$ we get that $\mathbb{L}_iq(X)\cong \mathbb{L}_{i-s}q(G^s(X))=0$.

Consider the exact complex $(Q^\bullet,d^\bullet)$ and let $X^i:=\Im (d^{i})$. Then for any $i\in \mathbb{Z}$
and any $j>0$,
we have $H^i(q(Q^\bullet))=\mathbb{L}_jq(X^{i+j})$ by the projective resolution
$$\cdots \rightarrow Q^i\rightarrow Q^{i+1}\rightarrow \cdots  Q^{i+j-1}\rightarrow  Q^{i+j}\rightarrow X^{i+j}\rightarrow 0.$$
Since $X^{i+j}\in \Gproj\mathcal{A}$, it follows from the above paragraph that $\mathbb{L}_jq(X^{i+j})=0$ for any $j>s$.
Therefore, for any $i\in \mathbb{Z}$, we may choose some $j>s$ and then $H^i(q(Q^\bullet))=\mathbb{L}_jq(X^{i+j})=0$.
That is, the complex $q(Q^\bullet)$ is acyclic, and it only remains to show
$\Hom _{\mathcal{A}}(Q^\bullet , i(P))$ is exact for
every $P \in \Proj \mathcal{B}$. For this, consider
the following short exact sequences arising from (\ref{Gn})
\begin{equation}\label{1}
0\rightarrow G(i(P))\rightarrow le(i(P))\rightarrow i(P)\rightarrow 0,
\end{equation}
\begin{equation}\label{2}
 0\rightarrow G^2(i(P))
\rightarrow lFe(i(P))\rightarrow G(i(P))\rightarrow 0,
\end{equation}
$$\cdots$$
\begin{equation}\label{3}
0\rightarrow 0=G^s(i(P))
\rightarrow lF^{s-1}e(i(P))\rightarrow G^{s-1}(i(P))\rightarrow 0.
\end{equation}
Since $ei \cong \Id_{\mathcal{B}}$
and $\Hom _{\mathcal{A}}(Q^\bullet , lF^j(P)) $
is exact,
we get that the complex $\Hom _{\mathcal{A}}(Q^\bullet , G^{s-1}(i(P)))$ is acyclic
by (\ref{3}). And so on, $\Hom _{\mathcal{A}}(Q^\bullet , G(i(P)))$ is acyclic
by (\ref{2}). Since $Q^\bullet$ is a complete $\mathcal{A}$-projective resolution
and $l(P)\in\Proj\mathcal{B}$,
we have that $\Hom _{\mathcal{A}}(Q^\bullet , l(P))$ is acyclic,
and by (\ref{1}), we get that $\Hom _{\mathcal{A}}(Q^\bullet , i(P))$ is also acyclic.
\end{proof}

We refer to the above conditions as the compatibility conditions on the cleft extension.
Our choice of the name ``compatible'' is motivated by the connection with the notion
of compatible bimodules, see Proposition~\ref{tri-matri-comp}.
\begin{definition}{\rm \label{def-comp-exten}
A cleft extension is
{\it compatible} if the conditions in Proposition~\ref{proposition-l-preserve-Gproj}
and Proposition~\ref{proposition-q-preserve-Gproj}
are satisfied.
}
\end{definition}

Applying Proposition~\ref{proposition-l-preserve-Gproj}
and Proposition~\ref{proposition-q-preserve-Gproj}, we get the following theorem immediately.
\begin{theorem}\label{thm-comp-exten}
Let $(\mathcal{B},\mathcal{A},
e,i,l)$ be a cleft extension which is compatible. Then
both $l$ and $q$ preserve Gorenstein projective objects.
\end{theorem}

Now we will compare Theorem~\ref{thm-comp-exten}
with \cite[Corollary 7.5]{Kos24}. Let's begin with the definition of perfect
functor from \cite{Kos24}.

\begin{definition}{\rm (\cite{Kos24})\label{def-ferfect-func}
A functor $F:\mathcal{B}\rightarrow \mathcal{B}$ is {\it perfect}
if the following conditions are satisfied:

(1) $\mathbb{L}_iF^j(F(P))=0$ for any $i,j>0$ and any $P\in \Proj \mathcal{B}$;

(2) there is some $n \geq 0$ such that $\mathbb{L}_iF^j= 0$ for all $i, j >0$ with $i + j \geq n + 1$;

(3) $\pd _{\mathcal{B}}F(P) < \infty$ for every $P \in \Proj \mathcal{B}$.

}
\end{definition}

\begin{proposition}\label{proposition-perfect}
Let $(\mathcal{B},\mathcal{A},
e,i,l)$ be a cleft extension of abelian categories. If the functor $F$
is perfect and nilpotent, then the cleft extension $(\mathcal{B},\mathcal{A},
e,i,l)$ is compatible.
\end{proposition}
\begin{proof}
Let $(P^\bullet,d^\bullet)$ be a complete $\mathcal{B}$-projective resolution with $X^i:=\Im (d^{i})$.
Then for any $i\in \mathbb{Z}$
and any $j>0$,
we have $H^i(F(P^\bullet))=\mathbb{L}_jF(X^{i+j})$ by the projective resolution
$$\cdots \rightarrow P^i\rightarrow P^{i+1}\rightarrow \cdots  P^{i+j-1}\rightarrow  P^{i+j}\rightarrow X^{i+j}\rightarrow 0.$$
Since $F$ is perfect,
there is some $n \geq 0$ such that $\mathbb{L}_jF= 0$ for any $j \geq n$.
Therefore, we may choose some $j\geq n$ and then $H^i(F(P^\bullet))=\mathbb{L}_jF(X^{i+j})=0$
for any $i\in \mathbb{Z}$, and then the complex $F(P^\bullet)$ is acyclic.
Since $F$ is perfect, we have $\pd _{\mathcal{B}}F(P') < \infty$ for any $P'\in \Proj \mathcal{B}$.
Then the complex $\Hom _{\mathcal{B}}(P^\bullet , F(P'))$ is acyclic as its cohomologies are
$\Ext  _{\mathcal{B}}^i(X, F(P'))$ for some $i>0$ and $X \in \GProj\mathcal{B}$.
Therefore, the assumption in Proposition~\ref{proposition-l-preserve-Gproj} is satisfied.

Let $(Q^\bullet,d^\bullet)$ be a complete $\mathcal{A}$-projective resolution with $Y^i:=\Im (d^{i})$.
Consider the acyclic complex $(e(Q^\bullet),e(d^\bullet))$ with $e(Y^i)=\Im (e(d^{i}))$.
Since $F$ is perfect, we have that
$\mathbb{L}_iF^j(F(P))=0$ for any $i,j>0$ and $P\in \Proj \mathcal{B}$, and by
\cite[Lemma 3.6]{Kos24}, we infer that $\mathbb{L}_iF^j(e(Q^l)))=0$ for any $i,j>0$ and $l\in \mathbb{Z}$.
Therefore, the complex $F^j(eQ^\bullet)$
computes $\mathbb{L}_iF^j(e(Y^l))$, that is, $H^{i}(F^j(e(Q^\bullet))) \cong\mathbb{L}_kF^j(e(Y^{i+k}))$
for any $i\in \mathbb{Z}$
and any $j,k>0$.
Since $F$ is perfect, it follows that
$\mathbb{L}_kF^j= 0$ for any $k + j \geq n + 1$.
Therefore, for any $i\in \mathbb{Z}$ and any $j>0$, we may choose sufficiently large $k$ such that
$H^{i}(F^j(e(Q^\bullet)))\cong\mathbb{L}_kF^j(e(Y^{i+k}))=0$,
and then the complex $F^j(eQ^\bullet)$ is exact.

Since $F$ is perfect, we have $\pd _{\mathcal{B}}F(P) < \infty$ and $\mathbb{L}_iF(F(P))=0$
for every $P \in \Proj \mathcal{B}$ and $i>0$. Let
$$0 \rightarrow P^t\rightarrow P^{t+1}\rightarrow \cdots  P^{-1}\rightarrow  P^{0}\rightarrow F(P)\rightarrow 0$$
be a projective resolution of $F(P)$. Applying the functor $F$, we get an exact sequence
$$0 \rightarrow F(P^t)\rightarrow F(P^{t+1})\rightarrow \cdots  F(P^{-1})\rightarrow  F(P^{0})\rightarrow F^2(P)\rightarrow 0,$$
which shows that $\pd _{\mathcal{B}}F^2(P) < \infty$ since $\pd _{\mathcal{B}}F(P^i) < \infty$ for any $i$.
On the other hand, it follows from \cite[Lemma 3.2]{Kos24} that $\mathbb{L}_iF(F^j(P))=0$
for every $i,j>0$. Using the same method and by induction, we have that $\pd _{\mathcal{B}}F^j(P) < \infty$
for any $j>0$. Since $\mathbb{L}_iF(F^j(P))=0$, it follows from \cite[Lemma 2.2 (2)]{Kos24}
that $\mathbb{L}_il(F^j(P))=0$ for any $i,j>0$. Applying the functor $l$ to the finite projective resolution of
$F^j(P)$, we deduce that
$\pd _{\mathcal{A}}l(F^j(P)) < \infty$ for any $j>0$. Then the complex $\Hom _{\mathcal{A}}(Q^\bullet , lF^j(P))$ is acyclic as its cohomologies are
$\Ext  _{\mathcal{A}}^i(X, lF^j(P))$ for some $i>0$ and $X \in \GProj\mathcal{A}$.
Therefore, the assumption in Proposition~\ref{proposition-q-preserve-Gproj} (ii) is satisfied.
\end{proof}

Applying Proposition~\ref{proposition-perfect} and Theorem~\ref{thm-comp-exten}, we recover \cite[Corollary 7.5]{Kos24}
without the
condition that $\pd _{\mathcal{B}}F(P) \leq d$.

\begin{corollary}\label{cor-pres-Gproj}{\rm (Compare \cite[Corollary 7.5]{Kos24})}
Let $(\mathcal{B},\mathcal{A},
e,i,l)$ be a cleft extension of abelian categories. If $F$ is perfect and nilpotent,
then both $l$ and $q$ preserve Gorenstein projective objects, and $X \in \GProj\mathcal{B}$
if and only if $l(X) \in \GProj\mathcal{A}$.
\end{corollary}
\begin{proof}
The first statement follows from Proposition~\ref{proposition-perfect}
and Theorem~\ref{thm-comp-exten}, and the second holds since $ql \cong \Id_{\mathcal{B}}$.
\end{proof}

\begin{remark}\label{remark-comp-exten}
{\rm In view of Theorem~\ref{thm-comp-exten}, Proposition~\ref{proposition-perfect}
and Corollary~\ref{cor-pres-Gproj}, we see that
in order to restrict $q$ and $l$ to the subcategories of Gorenstein projective objects,
our compatibility conditions
on cleft extension are more general than that of
\cite[Corollary 7.5]{Kos24}. Indeed,
the converse of Proposition~\ref{proposition-perfect}
is not true. That is, the functor $F$ in a compatible cleft extension may not be perfect.
Let's illustrate this with the following example.}
\end{remark}

\begin{example}\label{ex-tri-mar} {\rm To keep the notations in \cite{Z13},
we just consider the finite generated version. That is,
in this example and in Proposition~\ref{tri-matri-comp},
all modules over an Artin ring $A$ are finite generated, and all complete projective $A$-resolutions are
exact complex of finite generated projective $A$-modules which remain exact applying $\Hom _A(-,P)$
for any finite generated projective $A$-module $P$.
Let $\Lambda =
\left[\begin{array}{cc} A & _AM_B \\ 0 & B  \end{array}\right] $
be a triangular matrix Artin algebra.
A finite generated left $\Lambda$-module is identified with a triple
$(X,Y,f)$, where $X\in{\mod A},Y\in{\mod B}$, and $f:M\otimes_{B}Y\rightarrow{X}$ is a morphism of left $A$-modules.
It follows from \cite[Example 5.1]{Kos24} that
there is a cleft extension
$$\xymatrix@!=8pc{ \mod (A\times B) \ar[r]|{i} & \mod \Lambda
			 \ar[r]|{e} \ar@/_2pc/[l]|{q} & \mod (A\times B)
			\ar@/_2pc/[l]|{l} \ar@(ur,ul)_F    },  $$
where $q(X,Y,f)=(\Coker f,Y)$,
$i(X,Y)=(X,Y,0)$, $l(X,Y)=(X\oplus M\otimes Y,Y, \tiny {\left(\begin{array}{cc} 0 \\ 1  \end{array}\right)})$,
$e(X,Y,f)=(X,Y)$ and $F(X,Y)=(M\otimes _BY,0)$.
From \cite{Z13}, $_AM_B$ is compatible if $M\otimes _B-$ sends every acyclic complex of projective
$B$-modules to acyclic complex, and $\Hom _A(-,M)$ sends any complete $A$-projective resolution to acyclic complex.
In the following we will claim that
if $_AM_B$ is a compatible bimodule,
then the cleft extension $(\mod (A\times B), \mod \Lambda, i,e,l)$ is compatible.
In particular, if $\pd M_B<\infty$ and $\id _AM<\infty$, then it follows from \cite[Proposition 1.3]{Z13}
that $_AM_B$ is compatible and then the cleft extension $(\mod (A\times B), \mod \Lambda, i,e,l)$ is compatible.
But in this case the functor $F$ may not be perfect,
because this occurs precisely when $\fd M_B<\infty$ and $\pd _AM<\infty$, see
\cite[Example 5.1]{Kos24}.
}
\end{example}

\begin{proposition}\label{tri-matri-comp}
Keep the notations in Example~\ref{ex-tri-mar}. If $_AM_B$ is a compatible bimodule,
then the cleft extension $(\mod (A\times B), \mod \Lambda, i,e,l)$ is compatible.
\end{proposition}
\begin{proof}
Let $(P^\bullet, \widetilde{P}^\bullet)$
be a complete $A\times B$-projective resolution. Then $P^\bullet$
is a complete $A$-projective resolution and $\widetilde{P}^\bullet$ is a complete $B$-projective resolution.
Therefore, the complex $F(P^\bullet, \widetilde{P}^\bullet)
=(M\otimes _B\widetilde{P}^\bullet,0)$ is acyclic by the definition of compatible bimodule. Since every projective object
in $\mod (A\times B)$ is a direct summand of some copy of $(A,B)$, and $$\Hom _{A\times B}((P^\bullet, \widetilde{P}^\bullet),
F(A,B))=\Hom _{A\times B}((P^\bullet, \widetilde{P}^\bullet),
(M,0))\cong\Hom _A(P^\bullet, M),$$
we have that
the complex $\Hom _{A\times B}((P^\bullet, \widetilde{P}^\bullet),
F(A,B))$
is acyclic and thus the assumption in Proposition~\ref{proposition-l-preserve-Gproj} is satisfied.
On the other hand, we notice that $F^2=0$, and thus $F$ is nilpotent.
Since $F(A,B)=(M,0)$, so, in order to compute $\mathbb{L}_iF(F(A,B))$,
we begin with a projective
resolution of $(M,0)$ and apply $F$.
But all terms of this projective
resolution are direct summand of some copy of $(A,0)$, and then the fact $F(A,0)=0$
implies that $\mathbb{L}_iF(F(A,B))=0$ for any $i>0$.
Let $L^\bullet$ be a complete $\Lambda$-projective resolution.
By \cite[p.76]{ARS95}, all indecomposable projective $\Lambda$-modules are exactly
$(P,0,0)$ and $(M\otimes_BQ,Q,\Id)$, where $P$ runs over indecomposable projective
$A$-modules, and $Q$ runs over indecomposable projective $B$-modules.
So we may assume that the complex $L^\bullet$ is of the form
$$\xymatrix@!=6.2pc{\cdots \ar[r]& (P^{i}\oplus M\otimes _BQ^{i}, Q^{i}, \tiny {\left(\begin{array}{cc} 0 \\ 1  \end{array}\right)} )
\ar[rr]^(.45){(\varphi ^{i}, d^{i})}  & &
(P^{i+1}\oplus M\otimes _BQ^{i+1}, Q^{i+1}, \tiny {\left(\begin{array}{cc} 0 \\ 1  \end{array}\right)} )
\ar[r] & \cdots}, $$
and we deduce $\varphi ^{i}= \left(\begin{array}{cc} \delta ^{i} & 0 \\ \sigma ^{i} & \Id \otimes d^{i} \end{array}\right) $
form the commutative diagram
$$\xymatrix{M\otimes _BQ^i \ar[rr]^{ \tiny {\left(\begin{array}{cc} 0 \\ 1  \end{array}\right)}  }
\ar[d]_{\Id \otimes d^{i}}& & P^{i}\oplus M\otimes _BQ^{i} \ar[d]^{\varphi ^i}\\
M\otimes _BQ^{i+1} \ar[rr]^{ \tiny {\left(\begin{array}{cc} 0 \\ 1  \end{array}\right)}  }  & &P^{i+1}\oplus M\otimes _BQ^{i+1}. }$$
Therefore, the complex $F(eL^\bullet)$ is of the form
$$\xymatrix@!=4pc{\cdots \ar[r] & (M\otimes _BQ^i,0)
\ar[rr]^{(\Id \otimes d^{i},0)} & &(M\otimes _BQ^{i+1},0)\ar[r] & \cdots}.$$
Moreover, the exactness of $L^\bullet$ yeilds that both $(P^\bullet \oplus M\otimes _BQ^\bullet, \varphi^\bullet)$ and $(Q^\bullet, d^\bullet)$
are exact, and then $F(eL^\bullet)$ is acyclic since $M$ is compatible.
Note that $lF(A,B)=l(M,0)=(M,0,0)=i(M,0)$, and then we have
\begin{align*}
	\Hom _\Lambda(L^\bullet , lF(A,B)) &= \Hom _\Lambda(L^\bullet , i(M,0))\\
	&\cong \Hom _\Lambda(q(L^\bullet), (M,0)) \\
	&\cong \Hom _{A\times B}((P^\bullet, Q^\bullet), (M,0)) \\
    &\cong \Hom _{A}(P^\bullet, M),
    \end{align*}
where the first
isomorphism follows by adjunction.
Now we claim that $P^\bullet$ is a complete $A$-projective resolution,
and then $\Hom _\Lambda(L^\bullet , lF(A,B))\cong \Hom _{A}(P^\bullet, M)$ is acyclic and
the condition in Proposition~\ref{proposition-q-preserve-Gproj} is satisfied.
Indeed, there is an exact sequence of complexes of left $A$-modules
$$\xymatrix{0 \ar[r]& (M\otimes _BQ^\bullet, \Id \otimes d^{\bullet})
\ar[r]^{ \tiny {\left(\begin{array}{cc} 0 \\ 1  \end{array}\right)}  } & (P^{\bullet}\oplus
M\otimes _BQ^{\bullet}, \varphi^\bullet) \ar[r]^(.7){ \tiny {\left(\begin{array}{cc} 1 &0  \end{array}\right)}  } & (P^{\bullet},\delta ^\bullet)
\ar[r] & 0}. $$
and the exactness of $L^\bullet$ yeilds that both $(P^\bullet \oplus M\otimes _BQ^\bullet, \varphi ^\bullet)$ and $(Q^\bullet, d^\bullet)$
are exact. Therefore, we get that $(M\otimes _BQ^\bullet, \Id \otimes d^{\bullet})$ is acyclic since $M$ is compatible,
and then
the complex $(P^{\bullet}, \delta ^\bullet )$ is also acyclic. On the other hand,
$\Hom _A(P^{\bullet}, A)=\Hom _{A\times B}((P^{\bullet},Q^{\bullet}), (A,0))
=\Hom _{A\times B}(q(L^{\bullet}), (A,0))$, which is isomorphic to $\Hom _{\Lambda}(L^{\bullet}, i(A,0))$
by adjunction. Since $L^\bullet$ is a complete $\Lambda$-projective resolution
and $i(A,0)=(A,0,0)$ is a projective $\Lambda$-module, we deduce that
$\Hom _{\Lambda}(L^{\bullet}, i(A,0))$ is acyclic and so is $\Hom _A(P^{\bullet}, A)$.

\end{proof}

Next, we will give some necessary conditions for an object of $\mathcal{A}$ to be Gorenstein projective,
and we show that these necessary conditions are also sufficient in some special case.
Let's begin with some notations in cleft extensions.
From (\ref{G}) and (\ref{F}), there are two short exact sequences
$$\xymatrix{0 \ar[r]& G
\ar[r]^{u} & le \ar[r]^{\lambda} &\Id_{\mathcal{A}} \ar[r] & 0 } $$
and $$\xymatrix{0 \ar[r]& F
\ar@<+0.8ex>[r]^{\nu} & el \ar[r] \ar@<+0.8ex>[l]^{\nu '} &\Id_{\mathcal{A}} \ar[r] & 0 }, $$
where $u$, $\lambda$, $\nu$ and $\nu'$ are natural transformations.
Then for any $X\in \mathcal{A}$, we have a long exact
sequence $$\xymatrix@!=1.5pc{\cdots \ar[rr] & &leG^2(X)
\ar[rr] \ar @{->>}[rd]_{\lambda _{G^2(X)}} &  & leG(X) \ar[rr] \ar @{->>}[rd]|{\lambda _{G(X)}} & &le(X) \ar[rr]^{\lambda _X} & & X \ar[r] & 0    \\
& & & G^2(X)\ar @{^{(}->}[ur]|{u _{G(X)}} & & G(X)\ar @{^{(}->}[ur]_{u _X} & & &}. $$
Applying the functor $q$, and using the fact that $ql\cong \Id _{\mathcal{B}}$ and $eG^i\cong F^ie$, we obtain the following complex
$$\xymatrix@!=1.5pc{\cdots \ar[rr] & &F^2e(X)
\ar[rr]^{\alpha _X} \ar @{->>}[rd]_{q(\lambda _{G^2(X)})} &  & Fe(X) \ar[rr]^{\beta _X} \ar @{->>}[rd]|{q(\lambda _{G(X)})} & &e(X) \ar[rr]^{q(\lambda _X)} & & q(X) \ar[r] & 0    \\
& & & q(G^2(X))\ar [ur]|{q(u _{G(X)})} & & q(G(X))\ar [ur]_{q(u _X)} & & &},$$
where $q(u _{G(X)})$ and $q(u _X)$ may not be monomorphic.

\begin{theorem}\label{theorem-Gpro-condition}
Let $(\mathcal{B},\mathcal{A},
e,i,l)$ be a cleft extension which is compatible. Take $X\in \mathcal{A}$
and consider
the following conditions:

{\rm (1)} $X \in \GProj\mathcal{A}$;

{\rm (2)} $q(X) \in \GProj\mathcal{B}$ and $q(u _X)$ is a monomorphism;

{\rm (3)} $q(X) \in \GProj\mathcal{B}$ and the sequence $\xymatrix{F^2e(X)
\ar[r]^{\alpha _X} & Fe(X) \ar[r]^{\beta _X} & e(X) } $ is exact.

Then we have {\rm (1)} $\Rightarrow$ {\rm (2)} $\Leftrightarrow$ {\rm (3)}.
Moreover, all three conditions are equivalent if the natural transformation $\eta : F^2\rightarrow F$ arising from the composition
$$\xymatrix{F^2(Y) \ar[r]^{\nu _{FY}}& elF(Y)
\ar[r]^{el(\nu _Y)} & elel(Y) \ar[r]^{e(\lambda _{lY})} & el(Y) \ar[r]^{\nu '_Y} & F(Y) } $$
is zero, where $Y$ is an object in $\mathcal{B}$.

\end{theorem}
\begin{proof}
(1) $\Rightarrow$ (2): Let $X \in \GProj\mathcal{A}$ and $(Q^\bullet,d^\bullet)$ be a complete $\mathcal{A}$-projective resolution
with $X=\Im (d^{0})$. Then it follows from Theorem~\ref{thm-comp-exten} that
$q(X) \in \GProj\mathcal{B}$.
Since $Q^i\in \Proj\mathcal{A}$,
the exact sequence $$\xymatrix{0 \ar[r]& G(Q^i)
\ar[r]^{u _{Q^i}} & le(Q^i) \ar[r]^{\lambda  _{Q^i}} &Q^i \ar[r] & 0 }$$
is split,
and then $q(u _{Q^i})$ is a monomorphism. Therefore, we have an exact sequence
of complexes
\begin{equation}\label{exact-comp}
\xymatrix{0 \ar[r]& qG(Q^\bullet)
\ar[r]^{q(u _{Q^i})} & qle(Q^\bullet) \ar[r]^{q(\lambda  _{Q^i})} &q(Q^\bullet) \ar[r] & 0 .}
\end{equation}
Note that $e$ is an exact functor, and then the complex $qle(Q^\bullet)\cong e(Q^\bullet)$ is acyclic.
Since the cleft extension is compatible, we have that
the complex $q(Q^\bullet)$ is acyclic, and then we get that $qG(Q^\bullet)$ is also acyclic by (\ref{exact-comp}).
Let $\varepsilon: X\rightarrow Q^1$ be the monomorphism arising from the acyclic complex $(Q^\bullet,d^\bullet)$.
Then the exactness of $qG(Q^\bullet)$ implies that $qG(\varepsilon)$ is a monomorphism. By the commutative diagram
$$ \xymatrix{qG(X)\ar[rr]^{q(u _X)}\ar[d]^{qG(\varepsilon)} & & qle(X)\cong e(X) \ar[d]^{e(\varepsilon)}\\
qG(Q^1)\ar[rr]^{q(u _{Q^1})}& & qle(Q^1)\cong e(Q^1), } $$
we infer that $e(\varepsilon) \circ q(u _X)$ is a monomorphism, and thus $q(u _X)$ is a monomorphism.

{\rm (2)} $\Leftrightarrow$ {\rm (3)}: This is well-known. Here, we give a proof for readers' convenience.

If $q(u _X)$ is a monomorphism, then $\Ker (\beta _X)=\Ker (q(u _X) \circ q(\lambda _{G(X)}))
=\Ker(q(\lambda _{G(X)}))$, which is equal to $\Im (\alpha _X)$ since the sequence
$$ \xymatrix{F^2e(X)
\ar[r]^{\alpha _X} & Fe(X) \ar[r]^{q(\lambda _{G(X)})} & q(G(X)) \ar[r] &0}  $$
is exact. This shows that the sequence $\xymatrix{F^2e(X)
\ar[r]^{\alpha _X} & Fe(X) \ar[r]^{\beta _X} & e(X) } $ is exact.

Conversely, assume that the sequence $\xymatrix{F^2e(X)
\ar[r]^{\alpha _X} & Fe(X) \ar[r]^{\beta _X} & e(X) } $ is exact. Then $\Ker (\beta _X)=\Im (\alpha _X)
=\Ker(q(\lambda _{G(X)}))$. Let $\kappa:K\rightarrow Fe(X)$ be a morphism such that $\kappa=\Ker (\beta _X)=\Ker(q(\lambda _{G(X)}))$.
In order to show $q(u _X)$ is a monomorphism, we assume that
$q(u _X)\circ h=0$ for some $h:U\rightarrow q(G(X))$, and then we claim that $h=0$.
For this, consider the following commutative diagram
$$\xymatrix@!=1.5pc{K
\ar[rr]^{\kappa}  &  & Fe(X) \ar[rr]^{\beta _X} \ar @{->>}[rd]|{q(\lambda _{G(X)})} & &e(X)   \\
 & E\ar [ur]_{\gamma} \ar @{->>}[rd]_{\delta} \ar @{.>}[lu]^\theta & & q(G(X))\ar [ur]_{q(u _X)} &\\
 & &U \ar [ur]_{h} & &},$$ where $(E,\gamma,\delta)$ is the pull-back of $h$ and $q(\lambda _{G(X)})$.
Then $\delta$ is an epimorphism and $\beta _X \circ \gamma =q(u _X) \circ q(\lambda _{G(X)}) \circ \gamma
 =q(u _X) \circ h \circ \delta =0$. Therefore, there is a morphism $\theta: E\rightarrow K$
 such that $\gamma = \kappa \circ \theta$, and then $h \circ \delta=q(\lambda _{G(X)}) \circ \gamma=
 q(\lambda _{G(X)}) \circ \kappa \circ \theta$, which is equal to zero since $\kappa=\Ker(q(\lambda _{G(X)}))$.
Consequently, we infer that $h=0$ since $\delta$ is epimorphic.

{\rm (2)} $\Rightarrow$ {\rm (1)}: Assume $X \in \mathcal{A}$
such that $q(X) \in \GProj\mathcal{B}$ and $q(u _X)$ is a monomorphism.
If $\eta =0$, then it follows from \cite[Section 5.2]{Bel00}
that $\alpha _X=F(\beta _X)$. Therefore, we have that $qG(X)\cong \Coker(\alpha _X)=
\Coker(F(\beta _X))\cong F(\Coker(\beta _X))=Fq(X)$, where the last isomorphism
holds since the functor $F$ is right exact. Consequently, we have an exact sequence
\begin{equation}\label{exact-trivial-exten}
\xymatrix{0 \ar[r]& Fq(X)
\ar[r]^{q(u _X)} & e(X) \ar[r]^{q(\lambda _X)} & q(X) \ar[r] & 0 .}
\end{equation}
By \cite[Theorem 2.6]{Bel00}, there is an equivalence of categories $\mathcal{A}\cong \mathcal{B}_F(0)$,
where the category $\mathcal{B}_F(0)$ is defined as follows:
the objects in $\mathcal{B}_F(0)$ are of the forms $(X,f)$ where $X\in \mathcal{B}$
and $f:F(X)\rightarrow X$ is
a morphism in $\mathcal{B}$ with $f \circ F(f)=0$, and a morphism $a:(X,f)\rightarrow (Y,g) $ in
$\mathcal{B}_F(0)$ is a morphism $a:X\rightarrow Y$ in $\mathcal{B}$ with
$a\circ f=g\circ F(a)$. Therefore, we are in the situation of trivial extension of abelian categories.
Let $(P^\bullet,d^\bullet)$ be a complete $\mathcal{B}$-projective resolution
with $q(X)=\Im (d^{0})$. Since the cleft extension is compatible,
we get that
the complex $\Hom _{\mathcal{B}}(P^\bullet , F(P'))$ is acyclic for any $P' \in \Proj\mathcal{B}$,
and thus $\Ext _{\mathcal{B}}^1(\Ker d^i,F(P'))=0$. Applying the generalized horseshoe lemma of abelian categories
(\cite[Lemma 2.1]{PZH22}) to (\ref{exact-trivial-exten}),
and using the same method as in \cite[Theorem 3.3]{M25}, we obtain a complete $\mathcal{A}$-projective resolution
$$\xymatrix{\Delta: \cdots \ar[r]& l(P^0)
\ar[r]^{g^i} & l(P^1) \ar[r] & \cdots}$$
with $X=\Im (g^{0})$, and then $X \in \GProj\mathcal{A}$.
\end{proof}

\section{Applications and examples}
\indent\indent
In this section, we will give some applications
in triangular matrix rings, Morita context rings, trivial ring extensions
and $\theta$-extensions.
\begin{definition}
{\rm Let $R$ be a ring, $M$ a $R$-$R$-bimodule and $\theta : M \otimes _{R} M \rightarrow M $ an
associative $R$-bimodule homomorphism. The {\it $\theta$-extension} of $R$ by $M$, denoted by
$R\ltimes _{\theta} M$, is defined to be the ring with underlying group $R \oplus M$ and multiplication
given as follows: $$(r, m)\cdot(r',m'):=(r r',r m'+mr'+\theta (m\otimes m')),$$
for any $r, r'
\in R$ and $m, m' \in M$.}
\end{definition}
Let $T:=R\ltimes _{\theta} M$.
Then a left $T$-module is identified with a pair
$(X,\alpha)$, where $X\in \Mod R$ and $\alpha \in \Hom _R(M\otimes_{R}X, X)$ such that $\alpha \circ (1_M\otimes \alpha)
=\alpha\circ(\theta \otimes 1_X)$.
Moreover, we have the following ring homomorphisms
$R\rightarrow T$ given by $r \mapsto (r, 0)$ and $T\rightarrow R$ given by $(r, m) \mapsto r$.
By \cite[Section 6.3]{KP25}, we have the following
cleft extension
of module categories
$$\xymatrix@!=8pc{ \Mod R \ar[r]|{i=_TR\otimes _R-} & \Mod T
			 \ar[r]|{e=_RT\otimes _T-} \ar@/_2pc/[l]|{q=_RR\otimes _T-} & \Mod R
			\ar@/_2pc/[l]|{l=_TT\otimes _R-} \ar@(ur,ul)_F    },  $$
where $q(X,\alpha)=\Coker \alpha$,
$i(Y)=(Y,0)$, $l(Y)=(Y\oplus M\otimes _R Y,\tiny {\left(\begin{array}{cc} 0 &0 \\ 1 & \theta \otimes 1_Y \end{array}\right)})$,
$e(X,\alpha)=X$ and $F(Y)=M\otimes _RY$.
Moreover, every cleft extension of module categories
is isomorphic to a cleft extension induced by a
$\theta$-extension, see \cite[Proposition 6.9]{KP25}.

Consider the conditions in Proposition~\ref{proposition-l-preserve-Gproj} and Proposition~\ref{proposition-q-preserve-Gproj}.
Then the cleft extension induced by $T=R\ltimes _{\theta} M$ is compatible if the following two conditions hold:

(1) for any complete $R$-projective resolution $P^\bullet$,
both the complexes $M\otimes _RP^\bullet$ and $\Hom _{R}(P^\bullet , \Add M) $ are acyclic;

(2) both $R\otimes _TQ^\bullet$ and $\Hom _{T}(Q^\bullet , \Add R) $ are acyclic for any complete $T$-projective resolution $Q^\bullet$;
or $M$ is nilpotent, $\Tor _i^R(M^j, \Add M)=0$ for any $i,j>0$,
and both $M^j\otimes _Re(\widetilde{Q}^\bullet)$ and $\Hom _{T}(\widetilde{Q}^\bullet , \Add(T\otimes _RM^j)) $ are acyclic,
for any complete $T$-projective resolution $\widetilde{Q}^\bullet$.

Applying Theorem~\ref{thm-comp-exten} and Theorem~\ref{theorem-Gpro-condition} to $\theta$-extension,
we have the following corollary, which is a generalization of \cite[Corollary 3.6]{M25})
since a trivial ring extension $R\ltimes M$ is a $\theta$-extension with $\theta=0$.
\begin{corollary}\label{cor-theta-extension}{\rm (compare \cite[Corollary 3.6]{M25})}
Let $T:=R\ltimes _{\theta} M$ be a $\theta$-extension which satisfies the above two conditions. Then
both $l=T\otimes _R-$ and $q=R\otimes _T-$ preserve Gorenstein projective modules. In particular,
if $(X,\alpha)\in \GProj T$, then $\Coker \alpha \in \GProj R$ and the complex
$$\xymatrix@!=4pc{ M\otimes _RM\otimes _RX \ar[rr]^{1_M\otimes \alpha-\theta \otimes 1_X} & & M\otimes _RX \ar[r]^{\alpha}
&X}$$ is exact; and the converse also holds if $\theta =0$.
\end{corollary}
\begin{proof}
The first statement follows from Theorem~\ref{thm-comp-exten} immediately. For the second
one, we use Theorem~\ref{theorem-Gpro-condition} and the fact that $G(X,\alpha)=(M\otimes _RX, \theta \otimes 1_X-1_M\otimes \alpha)$,
$le(X, \alpha)=(X\oplus M\otimes _RX, \tiny {\left(\begin{array}{cc} 0 &0 \\ 1 & \theta \otimes 1_X \end{array}\right)}$),
$u_{(X, \alpha)}=\tiny {\left(\begin{array}{cc} -\alpha \\ 1  \end{array}\right)}$ and $\lambda_{(X, \alpha)}=(1, \alpha)$.
\end{proof}

Next, we will consider Morita context rings with zero homomorphisms and triangular matrix algebras,
which are two special examples of trivial ring extensions.
Let $A$ and $B$ be rings, $_AM_B$, $_BN_A$ two bimodule, and
$\Lambda _{(0,0)}=\left[\begin{array}{cc} A & _AM_B \\ _BN_A & B  \end{array}\right] $
be a Morita context ring with zero homomorphisms. Then a left $\Lambda_{(0,0)}$-module is identified with a quadruple
$(X,Y,f, g)$, where $X\in{\Mod A},Y\in{\Mod B}$, $f\in \Hom _B(N\otimes_{A}X, Y)$,
$g \in \Hom _A(M\otimes_{B}Y, X)$ such that $g\circ (1_M\otimes f)=0$ and $f\circ (1_N\otimes g)=0$.
On the other hand, $\Lambda _{(0,0)}=(A\times B)\ltimes (M\oplus N)$ is a trivial ring extension.
Therefore, the corresponding cleft extension is compatible if the following two conditions hold:

(1) for any complete $A$-projective resolution $P^\bullet$ and complete $B$-projective resolution $\widetilde{P}^\bullet$,
the complexes $N\otimes _AP^\bullet$, $M\otimes _B\widetilde{P}^\bullet$,
$\Hom _{A}(P^\bullet , \Add M) $ and $\Hom _{B}(\widetilde{P}^\bullet , \Add N) $ are acyclic;

(2) both $(A,B,0,0)\otimes _TQ^\bullet$ and $\Hom _{T}(Q^\bullet , \Add (A,B,0,0)) $ are acyclic for any complete $T$-projective resolution $Q^\bullet$.

Applying Theorem~\ref{theorem-Gpro-condition}, we establish some necessary and
sufficient conditions of infinitely generated Gorenstein-projective modules over Morita context rings,
which unify some known results such as \cite[Theorem 5.2]{M25} and \cite[Theorem 1]{Ase22}.
We mention that
the necessary and sufficient conditions were investigated in
\cite[Proposition 3.14]{GX24} for finitely generated Gorenstein projective modules over noetherian Morita context rings.

\begin{corollary}\label{cor-Morita-ring}{\rm (compare \cite[Theorem 5.2]{M25} and \cite[Proposition 3.14]{GX24})}
Let $\Lambda _{(0,0)}=\left[\begin{array}{cc} A & _AM_B \\ _BN_A & B  \end{array}\right] $
be a Morita context ring which satisfies the above two conditions. Then the following
statements are equivalent:

{\rm (1)} $(X,Y,f, g) \in \GProj \Lambda _{(0,0)}$;

{\rm (2)} $\Coker f \in \GProj B$, $\Coker g \in \GProj A$, and there are two induced isomorphisms
$M\otimes _B \Coker f\cong \Im g$ and $N\otimes _A \Coker g\cong \Im f$;

{\rm (3)} $\Coker f \in \GProj B$, $\Coker g \in \GProj A$,
and the following two sequences $$\xymatrix{M\otimes_{B}N\otimes _AX
\ar[r]^(.6){1_M\otimes f} & M\otimes_{B}Y \ar[r]^(.6){g} & X } $$
and $$\xymatrix{N\otimes_{A}M\otimes _BY
\ar[r]^(.6){1_N\otimes g} & N\otimes_{A}X \ar[r]^(.6){f} & Y }$$ are exact.
\end{corollary}
\begin{proof}
For any $(X,Y,f, g) \in \Mod \Lambda _{(0,0)}$, it follows that $$q(X,Y,f, g)=(\Coker g, \Coker f),$$
and the sequence in Theorem~\ref{theorem-Gpro-condition} (3) is of the form
$$\xymatrix@!=3pc{(M\otimes _BN\otimes_{A}X, N\otimes _AM\otimes_{B}Y)
\ar[rrr]^(.6){(1_M\otimes f, 1_N\otimes g)}  &  & &(M\otimes_{B}Y, N\otimes_{A}X) \ar[rr]^(.6){(g,f)} \ar @{->>}[rd] & &(X,Y)  \\
 &  & & &(M\otimes _B \Coker f, N\otimes _A \Coker g)\ar [ur]_{q(u_{(X,Y,f, g)})} }.$$
Note that $q(u_{(X,Y,f, g)})$ is a monomorphism if and only there are two isomorphisms $M\otimes _B \Coker f\cong \Im g$ and $N\otimes _A \Coker g\cong \Im f$
induced by $q(u_{(X,Y,f, g)})$.
Therefore, the statement follows from
Theorem~\ref{theorem-Gpro-condition}.
\end{proof}

Next, we will consider infinitely generated Gorenstein-projective modules over triangular matrix rings,
where the finitely generated version
was investigated in
\cite[Theorem 1.4]{Z13}.
\begin{corollary}\label{cor-tri-alg}{\rm (compare \cite[Theorem 1.4]{Z13})}
Let $\Lambda =
\left[\begin{array}{cc} A & _AM_B \\ 0 & B  \end{array}\right] $
be a triangular matrix ring. Assume that $M\otimes _BQ^\bullet$ is exact
for any exact sequence of infinitely generated projective $B$-modules, and
$\Hom _{A}(P^\bullet , \Add M) $ is exact
for any complete $A$-projective resolution $P^\bullet$ in $\Mod A$.
Then a $\Lambda$-module $(X,Y,g)\in\GProj \Lambda $ if and only if
$\Coker g \in \GProj A$, $Y\in \GProj B$ and $g$ is a monomorphism.
\end{corollary}

\begin{proof}
By the proof of Proposition~\ref{tri-matri-comp}, we deduce that the corresponding
cleft extension of infinitely generated modules is compatible,
and then we have the equivalent descriptions of the infinitely generated Gorenstein-projective module
as in Corollary~\ref{cor-Morita-ring} where $N=0$.
Moreover, the sequences in Corollary~\ref{cor-Morita-ring} (3)
are exact if and only if $g$ is a monomorphism.
Therefore, this statement follows from Corollary~\ref{cor-Morita-ring}.
\end{proof}

Now we will explain some related results on tensor rings. Let $R$ be a ring and $M$ a $R$-$R$-bimodule.
Recall that the {\it tensor ring} $T_R(M) =\oplus _{i=0}^{\infty}M^i$,
where $M^0 = R$ and $M^{i+1} = M \otimes _R M^{i}$ for $i \geq 0$. It follows from
\cite[Example 6.8]{KP25} that $T_R(M)$ is a special $\theta$-extension with $\theta\neq 0$.
Moreover, a $T_R(M)$-module is identified with a pair
$(X,\alpha)$, where $X\in \Mod R$ and $\alpha \in \Hom _R(M\otimes_{R}X, X)$. Consider
the cleft extension $(\Mod R, \Mod T_R(M), i,e,l)$. We can check that $qG(X,\alpha)=M\otimes _RX$
and the morphism $q(u_{(X,\alpha)})$ in Theorem~\ref{theorem-Gpro-condition} is just $\alpha$.
Therefore, if $(X,\alpha) \in \GProj (T_R(M))$ then $\Coker \alpha \in \GProj R$ and $\alpha$ is monomorphic.
We mention that
the converse was proved true
in \cite[Theorem 3.9]{CL20} and \cite[Theorem A]{DLST25}.
However, we don't know whether the necessary conditions in Theorem~\ref{theorem-Gpro-condition}
are also sufficient when $\theta\neq0$, because the proof of \cite[Theorem 3.9]{CL20} and \cite[Theorem A]{DLST25}
rely on the fact that $qG(X,\alpha)=M\otimes _RX$, which is not true for general $\theta$-extensions
or cleft extensions.

\noindent {\footnotesize {\bf ACKNOWLEDGMENT.}
This work is supported by
the National Natural Science Foundation of China (12061060),
the project of Young and Middle-aged Academic and Technological leader of Yunnan
(202305AC160005) and the Basic Research Program of Yunnan Province (202301AT070070).}


\begin{thebibliography}{99}

\bibitem{Ase22} D. Asefa, Gorenstein-projective modules over a class of Morita rings, J. Math. 1 (2022), 3732360, 8 pp.

\bibitem{AB69} M. Auslander and M. Bridger, Stable Module Theory, Mem. Amer. Math. Soc., no. 94,
Amer. Math. Soc., Providence, R.I., 1969.


\bibitem{ARS95}  M. Auslander, I. Reiten and S.O. Smal{\o}, Representation
theory of Artin algebras, Cambridge Studies in Advanced Mathematics
36, Cambridge University Press, Cambridge, 1995.

\bibitem{Bel00} A. Beligiannis, Cleft extensions of abelian categories and applications to ring theory,
Comm. Algebra 28 (2000), 4503--4546.

\bibitem{CL20} X. W. Chen and M. Lu, Gorenstein homological properties of tensor rings, Nagoya
Math. J. 237 (2020), 188--208.

\bibitem{CR22} X. W. Chen and W. Ren, Frobenius functors and Gorenstein homological
properties, J. Algebra 610 (2022), 18--37.

\bibitem{DLST25} Z. X. Di, L. Liang, Z. Q. Song and G. L. Tang, Gorenstein homological modules over tensor rings,
arXiv:2504.21349.

\bibitem{EJ95} E. E. Enochs and O.M.G. Jenda, Gorenstein injective and projective modules, Math. Z. 220 (1995), 611--633.

\bibitem{EJ00} E. E. Enochs and O.M.G. Jenda, Relative Homological Algebra, de Gruyter Exp. Math., vol. 30, Walter de Gruyter and Co., 2000.

\bibitem{EGPS22} K. Erdmann, O. Giatagantzidis, C. Psaroudakis and {\O}. Solberg, Monomial arrow removal
and the finitistic dimension conjecture, arXiv:2506.23747.

\bibitem{EPS22} K. Erdmann, C. Psaroudakis and {\O}. Solberg, Homological invariants of the arrow removal operation, Represent.
 Theory 26 (2022), 370--387.

\bibitem{FGR75} R. M. Fossum, P. A. Griffith and I. Reiten, Trivial extensions of abelian categories,
 Lecture Notes in Mathematics, Vol. 456, Springer-Verlag, Berlin-New York, 1975.

\bibitem{GP17} N. Gao and C. Psaroudakis, Gorenstein homological aspects of monomorphism
categories via Morita rings, Algebr. Represent. Theory 20 (2017), no. 2, 487--529.

\bibitem{GX24} Q. Q. Guo and C. C. Xi, Gorenstein projective modules over rings of morita contexts,
Sci. China Math. 67 (2024), no. 11, 2453--2484.

\bibitem{GPS18} E. L. Green, C. Psaroudakis and {\O}. Solberg, Reduction techniques
for the finitistic dimension, Trans. Am. Math. Soc 374 (2021), 6839--6879.

\bibitem{H25} D. D. Hu, Cotorsion pairs on extensions of abelian categories, arXiv:2506.17898.

\bibitem{Kos24} P. Kostas, Cleft extensions of rings and singularity categories, arXiv:2409.07919v2.

\bibitem{KP25} P. Kostas and C. Psaroudakis, Injective generation for graded rings, J. Pure Appl. Algebra 229 (2025), no. 7, 107960, 40 pp.

\bibitem{LHZ22} H.  H. Li, J. S. Hu and Y. F. Zheng, When the Schur functor induces a triangle-equivalence between Gorenstein
defect categories, Sci China Math 65 (2022), 2019--2034.


\bibitem{LMY25} L. Liang, Y. J. Ma and G. Yang, Homological invariant properties under cleft extensions,
arXiv:2506.02691.

\bibitem{LZ20} Z. W. Li and P. Zhang, A construction of Gorenstein-projective modules, J. Algebra 323 (2010), 1802--1812.


\bibitem{Lu17} M. Lu, Gorenstein defect categories of triangular matrix algebras,
J. Algebra 480 (2017), 346--367.

\bibitem{Lu19} M. Lu, Gorenstein properties of simple gluing algebras, Algebr. Represent. Theor. 22 (2019),
517--543.

\bibitem{M25} L. X. Mao, Gorenstein projective, injective and fat modules over trivial ring extensions, J.
Algebra Appl. 24 (2025), no. 2, 2550030.

\bibitem{MZL25} Y. J. Ma, J. L. Zheng and Y. Z. Liu, Igusa-Todorov distances, extension and Rouquier dimensions under cleft
extensions of abelian categories, arXiv:2411.07804v4.


\bibitem{Mar93} N. Marmaridis, On Extensions of Abelian Categories with Applications to Ring Theory, J.
Algebra 156 (1993), no. 1, 50--64.

\bibitem{PZH22} Y. Y. Peng, R. M. Zhu and Z. Y. Huang, Gorenstein projective objects in comma categories, Period.
Math. Hung. 84 (2022), 186--202.

\bibitem{QS24} Y. Y. Qin and D. W. Shen, Eventually homological isomorphisms and Gorenstein projective modules, Sci China
Math. 67 (2024), 1719-1734.

\bibitem{XZ12} B. L. Xiong and P. Zhang, Gorenstein-projective modules over triangular matrix Artin algebras, J Algebra Appl 11 (2012),
1250066, 14 pp.

\bibitem{Z13} P. Zhang, Gorenstein-projective modules and symmetric recollements, J Algebra 388 (2013), 65--80.

\end{thebibliography}
\end{document}